\documentclass[11pt,a4paper]{amsart}

\usepackage{amsmath}
\usepackage{amssymb}

\theoremstyle{plain}   %% This is the default, anyway
\begingroup % Confine the \theorembodyfont command
\newtheorem{theorem}{Theorem}[section]   % Numbered within each section
\newtheorem{corollary}[theorem]{Corollary}     % Numbered along with thm
\newtheorem{lemma}[theorem]{Lemma}         % Numbered along with thm
\newtheorem{proposition}[theorem]{Proposition}  % Numbered along with thm

\newtheorem{question}[theorem]{Question}

\endgroup

\theoremstyle{definition}
\newtheorem{definition}[theorem]{Definition}  % Numbered along with thm
\newtheorem{remark}[theorem]{Remark}
\newtheorem{example}[theorem]{Example}
\newtheorem*{proofof}{Proof of Theorem \ref{Q}}

\numberwithin{equation}{section}

\newcommand{\ep}{\varepsilon}

\newcommand{\R}{{\mathbb R}}
\newcommand{\N}{{\mathbb N}}

\newcommand{\vf}{\varphi}

\newcommand{\diam}{\operatorname{diam}}

\newcommand{\M}{{\cal M}}

\newcommand{\cal}{\mathcal}

\newcommand{\rec}{{\rm rec}\,}
\newcommand{\spa}{\operatorname{span}}
\newcommand{\loc}{{\rm loc}\,}

\usepackage[dvips]{color}

\begin{document}

\title{Functions on a convex set which are both $\omega$-semiconvex and $\omega$-semiconcave}

\author{V\' aclav Kry\v stof}
\author{Lud\v ek Zaj\'\i\v cek}

\address{Charles University, Faculty of Mathematics and Physics, Sokolovsk\'a 83, 186 75 Praha 8, Czech Republic}

\email{krystof@karlin.mff.cuni.cz}
\email{zajicek@karlin.mff.cuni.cz}



\thanks{The research was supported by GA\v CR~18-11058S}

\begin{abstract}
Let $G \subset \R^n$ be an open convex set which is either bounded or contains a translation of
 a convex cone with nonempty interior. It is known that then, for every modulus $\omega$,
 every function on $G$ which is both semiconvex and semiconcave with modulus $\omega$ is (globally)
 $C^{1,\omega}$-smooth. We show that this result is optimal in the sense that the assumption on
 $G$ cannot be relaxed. We also present direct 
short proofs of the above  mentioned result and of some its quantitative versions.
 Our results have immediate consequences concerning (i) a first-order quantitative
 converse Taylor theorem  and  (ii) the problem whether 
$f\in C^{1,\omega}(G)$  whenever  $f$ is continuous and smooth in a corresponding sense on all lines.
 We hope that these consequences are of an independent interest.
\end{abstract}

\maketitle

\section{Introduction}

First we recall  a classical reasult on smoothness of functions which are both semiconvex and semiconcave (i.e.,  semiconvex and semiconcave with linear modulus). 
\begin{definition}\label{klsem}
Let $X$  be a Hilbert space, $G \subset X$ an open convex set and $f$ a function on $G$. We say that
 $f$ is semiconvex, if there exists $C\geq 0$ and a continuous convex function $g$ on $G$
such that
$$  f(x)= g(x) - \frac{C}{2} \cdot \|x\|^2,\ \ \ x \in G.$$
Then we say that $f$ is semiconvex with constant $C$.
We say that $f$ is semiconcave on $G$ (with constant $C$) if $-f$ is semiconvex on $G$  
 (with constant $C$).
\end{definition}
 Let $X$, $G$ and $f$ be as in Definition \ref{klsem}. Then it
 is very easy to show (see Remark \ref{lehimpl}) that  
\begin{multline}\label{kls}
f' \ \text{exists and is Lipschitz with constant}\ C	\\
\Longrightarrow\ \ \ 
f\ \text{ is both semiconvex and semiconcave with   constant}\ C.
	\end{multline}  
	  Rather suprisingly, also the converse of \eqref{kls} holds:
	\begin{multline}\label{kl1}
	f\ \text{ is both semiconvex and semiconcave with  constant}\ C  \\
\Longrightarrow\ \ \ 	f' \ \text{exists and is Lipschitz with  constant}\ C.
	\end{multline}
	Some bibliographic notes concerning this well-known interesting
 nontrivial  result are contained
 in Section \ref{lin}, where also its  short elementary 	proof  is presented.

	Consequently, we have that
	\begin{multline}\label{kl1e}
	f\ \text{ is both semiconvex and semiconcave with  constant}\ C  \\
\Longleftrightarrow\ \ \ 	f' \ \text{exists and is Lipschitz with  constant}\ C.
	\end{multline}
	In particular,
	\begin{equation}\label{kl2}
f\  \text{is both semiconvex and semiconcave} \ \Longleftrightarrow\ 
 f \in C^{1,1} (G).
\end{equation}

In the present article
 we study analogues of  \eqref{kl1}  and  \eqref{kl2} for functions
 which are both $\omega$-semiconvex and $\omega$-semiconcave. We use the following
 terminology.

\begin{definition}
We denote by $ \M $ the set of all $ \omega:[0,\infty)\to [0,\infty) $ which are non-decreasing and satisfy $ \lim_{t\to 0+}\omega(t)=0 $.
\end{definition}

\begin{definition}\label{omse} 
Let $ X $ be a normed linear space, $ G\subset X $ an open convex set and $ \omega\in\M $.
 We say that a continuous $ f:G\to\R $ is semiconvex with modulus $ \omega $ (or $\omega$-semiconvex for short) if
$$
f(\lambda x+(1-\lambda)y)\leq \lambda f(x)+(1-\lambda)f(y) + \lambda(1-\lambda)\Vert x-y \Vert \omega(\Vert x-y \Vert)
$$
for every $ x,y\in G $ and $ \lambda\in [0,1] $. We say that $ f:G\to\R $ is $\omega$-semiconcave  if $ -f $ is $\omega$-semiconvex.
\end{definition}
\begin{remark}\label{H}
If $X$ is a Hilbert space, it is easy to show (cf. \cite[p. 30]{CaSi})
 that $f$ is semiconvex on $G\subset X$ with constant $C\geq 0$ if and only if $f$ is 
 $\omega$-semiconvex with modulus $\omega(t)=\frac{C}{2}t,\ t\geq 0$.
\end{remark}
\begin{remark}\label{para}
The notion of $\omega$-semiconvex functions is very close to 
Rolewicz's notion of {\it strongly $\alpha(\cdot)$-paraconvex functions} (which has
 in \cite{Rol3} slightly different sense than in \cite{Rol2}). Namely 
(see e.g. \cite[Remark 2.11]{DZ1})
  if $\omega \in \M$ and
 $\alpha(t)= t\omega(t),\ t\geq 0$, then (by the definition from \cite{Rol3}),
$$\text{$f$ is $\omega$-semiconvex \ $\Longrightarrow$\ $f$ is  
 strongly $\alpha(\cdot)$-paraconvex\ $\Longrightarrow$\ 
 $f$ is $(2\omega)$-semiconvex.}$$
Thus most results conserning $\omega$-semiconvex functions have obvious
 consequences for  strongly $\alpha(\cdot)$-paraconvex functions (and vice versa).
 We will not  formulate such consequences of our results.
\end{remark}   

 \begin{definition}\label{cjo}
If $ G $ is an open subset of a normed linear space and $ \omega\in\M $, then we denote by $ C^{1,\omega}_*(G) $ the set of all Fr\' echet differentiable $ f:G\rightarrow\R $ such that $ f' $ is uniformly continuous with modulus $ \omega $.
We set $ C^{1,\omega}(G):= \bigcup_{K>0} C^{1,K\omega}_*(G) $.
\end{definition}
Note that  so defined  $ C^{1,\omega}$-smoothness, which coincides with (global)
$C^{1,1}$-smoothness in the case of a linear $\omega$, is common in the literature.
 We will sometimes use also  (stronger)  $ C_*^{1,\omega}$-smoothness (which
 is somewhere, e.g. in \cite{Jo} and \cite{JZ}, called ``$C^{1,\omega}$-smoothness'').
\smallskip

Let us first mention that a (weaker) version 
  of \eqref{kls} for general moduli is also easy
 (see Remark \ref{lehimpl}). It asserts that, if $X$, $G$ and $\omega$ are as in Definition \ref{omse}, then
\begin{equation}\label{limpl}
f \in C^{1,\omega}_*(G)
\Longrightarrow\ 
\text{$f$ is both $\omega$-semiconvex and $\omega$-semiconcave on}\ G.
\end{equation}
However (if $X$, $G$ and $\omega$ are as in Definition \ref{omse})
 the implication
\begin{equation}\label{impl}
f\  \text{is both $\omega$-semiconvex and $\omega$-semiconcave on $G$} \ \Longrightarrow\ 
 f \in C^{1,\omega} (G)
\end{equation}
does not hold in general.
The validity of \eqref{impl} (or of its local versions) was studied
 e.g. in  \cite[Theorem 3.3.7]{CaSi}, \cite{JTZ} and \cite{Kr}
 (and essentially already in \cite{Rol1} for $\omega(t)= t^{\alpha}$, $0< \alpha\leq 1$).
As far as we know, the strongest and the most general  versions of
 \eqref{impl} follow  from  versions of Converse Taylor Theorem proved in  \cite{Jo} and \cite{HJ}.
 In particular (see \cite[Theorem 2.6 and Remark 2.7]{Kr}),
  \cite[Corollary 126]{HJ} easily implies the following result.

\begin{theorem}\label{HJ}\cite{Kr}
Let $X$ be a normed linear space and $G\subset X$ an open convex set. Suppose that $G$ is either bounded or has the property that there are $a\in X$, $r>0$ and $\{u_n\}_{n\in \N} \subset X$,
 $\|u_n\| = n$, such that  $B(a+u_n, nr) \subset G$ for each $n \in \N$.
Then, for each modulus $\omega \in \M$ and for each  function  $f$ on $G$,
$$ f\  \text{is both $\omega$-semiconvex and $\omega$-semiconcave on}\ G \ \Longrightarrow\ 
 f \in C^{1,\omega} (G).$$
\end{theorem}
\begin{remark}\label{ekvom}
Using  \eqref{limpl}, we easily see that the assertion of Theorem \ref{HJ} implies the
 following analogue of \eqref{kl2}:
\begin{equation}\label{chardve}
 f\  \text{is both $(K\omega)$-semiconvex and $(K\omega)$-semiconcave}\ \text{for some}\ K>0 \ \Longleftrightarrow\ 
 f \in C^{1,\omega} (G).
\end{equation} 
For a local version of \eqref{chardve} see \cite[Theorem 6.1]{JTZ}. See also Proposition 
\ref{charfor}
 for reformulations of this characterization of $C^{1,\omega}$-smoothness on some convex sets. 
\end{remark}

Let us note that an easy compactness argument shows that an unbounded open convex $G \subset \R^n$
satisfies the assumption of Theorem \ref{HJ} if and only if it contains a translation of a convex cone with nonempty interior.

The main result of the present article is the following theorem which shows  that, for $X= \R^n$,  Theorem \ref{HJ} is optimal in the sense that the assumption on
 $G$ cannot be relaxed. It follows immediately from Proposition \ref{P} (proved in Section \ref{opt}) via Remark \ref{gimpl}.

\begin{theorem}\label{hlav}
Let $G \subset \R^n$ ($n\geq 2$) be an unbounded open convex set which does not contain a translation of a convex cone with nonempty interior. Then there exists a concave modulus $\omega \in \M$
 with $\lim_{t \to \infty} \omega(t) = \infty$ and a function $f$ which is both $\omega$-semiconvex and $\omega$-semiconcave on $G$ but $f \notin C^{1,\omega}(G)$. 
\end{theorem}

Since \cite[Corollary 126]{HJ} is a rather general result (working with classes  $C^{k,\omega}(G)$
 for each $k \in \N$), we  find useful to
  present (in Section \ref{gen})  direct
 proofs (which partly follow arguments from \cite{HJ}) of Theorem \ref{HJ} and of some its quantative versions (sharper then those  which directly follow from the proof of \cite[Corollary 126]{HJ}). Namely, we prove the following result.

\begin{theorem}\label{Q}
Let $X$ be a normed linear space and $\emptyset \neq G\subset X$ an open convex set. 
 Let $\omega \in \M$ and let $f$ be a function on $G$ which is both
 $\omega$-semiconvex and $\omega$-semiconcave. Then the following assertions hold.
\begin{enumerate}
\item[(i)]
If $G$ is bounded, then $f'$ is uniformly continuous with modulus
  $6 e_G \omega $, where  $e_G:=  \frac{\diam G}{ \sup\{r: B(a,r) \subset G\}}$. 
\item[(ii)]
If $a\in X$, $r>0$ and $\{u_n\}_{n\in \N} \subset X$ with
 $\|u_n\| = n$ are such that  $B(a+u_n, nr) \subset G$ for each $n \in \N$, then
  $f'$ is uniformly continuous with modulus  $12\, (1+1/r) \omega$.
\end{enumerate}
\end{theorem}
\begin{remark}\label{kon}
\begin{enumerate}
\item[(i)] The proof of \cite[Corollary 126]{HJ} gives 
  the multiplicative constants 15 and 30 instead of our constants
6 and 12.
\item[(ii)] In the special cases when $G=X$ or $\omega$ is linear,
 we can assert that $f$ is $C_*^{1, 4 \omega}$-smooth
  (see Corollary \ref{cepr} and Remark \ref{specimpl}).
\end{enumerate}
\end{remark}
In the last Section \ref{posl} we show that
our results have immediate consequences (or reformulations) concerning 
\begin{enumerate}
\item[(i)]
 a first-order quantitative 
 converse Taylor theorem  and
\item[(ii)]
 the problem whether 
$f\in C^{1,\omega}(G)$  whenever $f$ is continuous and $C^{1,\omega}_*$-smooth on all lines.
\end{enumerate}
{\it We hope that these versions of our new results can be interesting also for readers which are not interested in $\omega$-semiconvex functions.}

At the end of Section \ref{posl}, a natural open question is formulated.

\section{Preliminaries}

\subsection{Basic definitions}
In the following, $X$ will be always a real normed linear space and $X^*$ its dual space.
The norm and the origin in any normed linear space will be denoted
 by $\|\cdot\|$ and $0$, respectively. We consider real Hilbert spaces which can be finite-dimensional. By the derivative $f'$ of a
 function $f$ defined on a subset of $X$ we always mean the Fr\' echet derivative and the notion
 of $C^1$-smoothness has the standard sense. 
 The symbol $B(x,r)$ ($\overline B(x,r)$) denotes the open (closed) ball with centre $x$ and radius $r$ in $X$.
 For $A \subset \R^n$, the symbols 
$\overline A$ and $\spa A$ denote the closure and the linear span of $A$, respectively.
By a  cone in $\R^n$ we mean a set $C \subset \R^n$ such that $\lambda x \in C$ for each
 $x \in C$ and $\lambda >0$. A ray in $\R^n$ is a set of the form $\rho= \{ a+tv:\ t\geq 0\}$,
 where $a \in \R^n$ and $0\neq v \in \R^n$. By the dimension $\dim (C)$ of a convex set
 $C \subset \R^n$ we mean the dimension of the affine hull of $C$ (see \cite[p. 12]{Roc}).

\subsection{$\omega$-semiconvexity and $\omega$-semiconcavity}
We will frequently use the obvious fact that if $f$ is $\omega$-semiconvex (resp. $\omega$-semiconcave) on $G$ and $c>0$, then
 the function $cf$ is $(c\omega)$-semiconvex (resp. $(c\omega)$-semiconcave).

Further, it follows easily from Definition \ref{omse} that $\omega$-semiconvexity is equivalent to  ``$\omega$-semiconvexity on all lines''. More precisely, the following statement  holds.

\begin{lemma}\label{pp}
Let $X$ be a normed linear space, $G\subset X$ an open convex set,
 $\omega \in \M$, and let $f$ be a continuous function on $G$. Then
 the following conditions are equivalent.
\begin{enumerate}
\item[(i)] $f$ is $\omega$-semiconvex on $G$.
\item[(ii)] For every $a \in G$ and $v\in X$, $\|v\|=1$, the function
 $f_{a,v}: t \mapsto f(a+tv)$ defined on the open interval
 $\{t\in \R:\ a+tv \in G\}$ is $\omega$-semiconvex.
\end{enumerate}
\end{lemma}

\begin{remark}\label{lehimpl}
\hfill
\begin{enumerate}
\item[(i)]
Using Lemma \ref{pp}, it is easy to see that, to prove  implications \eqref{kls} and \eqref{limpl},
 it is sufficient to prove them for $X=\R$. The validity of \eqref{kls} for $X= \R$ is almost 
obvious. Indeed, the assumption of \eqref{kls} implies that $g(x):= f(x) + \frac{C}{2} x^2$
 is convex on $G$, since $g'(x) = f'(x) + Cx$ is clearly nondecreasing on the interval $G$.
And  the validity of \eqref{limpl} for $X=\R$ immediately follows
  e.g. from \cite[Proposition 2.1.2]{CaSi}.
	\item[(ii)]
	Using Lemma \ref{pp}, Remark \ref{H} and \eqref{kls}, we obtain that if $\omega$ is linear,
	 then we can assert in \eqref{limpl} even that $f$ is both $\tilde \omega$-semiconvex and
	 $\tilde \omega$-semiconcave, where $\tilde \omega:= \frac{1}{2} \omega$. However, for general modulus it is not true even for $X=\R$ and $\tilde \omega:= C \omega$, where $C<1$ is any absolute constant. Indeed, if $0<C<1$ is given, we choose $p \in (1,2)$ such that
	 $C< 1/p$ and set $\omega(t):= t^{p-1}$, $t\geq 0$,  and $f(x):= x^p/p,\ x \in (0,1)$.
	 Then $f'(x)= \omega(x),\ x\in (0,1),$ and consequently  the concavity of $\omega$ implies that $f$ is $C_*^{1,\omega}$-smooth. However (as observed  in \cite[ Remark 2.10 (ii)]{Kr2}), it is not difficult to show that
	 $f$ is not $(C\omega)$-semiconcave.
	\end{enumerate}
	\end{remark}
	We will need the following fact which is an easy consequence of \eqref{limpl}
	 and Lemma \ref{pp}.
	
	\begin{lemma}\label{post}
Let $X$ be a normed linear space,  $ G \subset X $  an open convex set, $ \omega\in\M $ and let  $f$ be a Fr\' echet differentiable function on $G$.
Suppose that
\begin{equation}\label{smder}
\vert f'(y)(y-x)-f'(x)(y-x)\vert\leq \|y-x\| \omega(\| y-x\|),\quad x,y\in G.
\end{equation}
Then $f$ is both $\omega$-semiconvex and $\omega$-semiconcave on $G$.
\end{lemma}

\begin{proof}
We will first observe that if $f_{a,v}$ is as any function as in  Lemma \ref{pp} (ii), then
\begin{equation}\label{hlnap}
\text{$f_{a,v}$ is $C^{1,\omega}_*$-smooth.}
\end{equation}
Indeed, for each $t_1 \neq t_2$ from
 the domain of $f_{a,v}$ we can apply \eqref{smder} to  $x:=a+t_1 v$ and $y:=a+t_2 v$, and obtain
\begin{multline*}
 |(f_{a,v})'(t_2) - (f_{a,v})'(t_1)|= |f'(a+t_2 v)(v) - f'(a+t_1 v)(v)|\\
= \frac{1}{|t_2-t_1|} |f'(a+t_2 v)((t_2-t_1)v) - 
f'(a+t_1 v)((t_2-t_1)v)| \leq \omega(|t_2-t_1|),
 \end{multline*}
Therefore each $f_{a,v}$ is $\omega$-semiconvex by \eqref{limpl} and so
 $f$ is $\omega$-semiconvex by Lemma \ref{pp}.
Using the same argument to $-f$, we obtain that 
$f$ is also
 $\omega$-semiconcave.
\end{proof}

An other basic tool for us is the following essentially well-known result.

\begin{lemma}\label{apr}
Let $X$ be a normed linear space,  $ G \subset X $  an open convex set and $ \omega\in\M $.
Let a function $f$ be both $\omega$-semiconvex and $\omega$-semiconcave on $G$. Then $f$ is Fr\' echet differentiable on $G$ and
\begin{equation}\label{ssapr}
|f(y)-f(x)- f'(x)(y-x)| \leq \|y-x\| \, \omega(\|y-x\|),\ \ \  x, y \in G.
\end{equation}
\end{lemma}

\begin{proof} 
A proof in the case $X=\R^n$ can be found in \cite{CaSi} (see the proof of
 \cite[(3.23), p. 60]{CaSi}). The proof of the general case is factually  contained
 in the proof of \cite[Theorem 2.6]{Kr}.  
For the convenience  of the reader, we shortly repeat the argument from \cite{Kr}.
 It is based on \cite[Theorem 3]{Rol2} (cf. also \cite[Proposition 2]{Rol3})
 which easily implies (cf. \cite{Kr} or \cite[Corollary 4.5]{DZ1}) that
 $\omega$-semiconvexity of $f$ implies that for each $x\in G$ there exists $\Phi_x \in X^*$
 such that 
 $$  
f(x+h) - f(x)- \Phi_x(h) \geq  - \|h\| \omega(\|h\|)\ \ \text{whenever}\ \ \ x+h \in G.$$
Since $-f$ is also $\omega$-semiconvex, for each $x\in G$ there exists $\Psi_x \in X^*$
 such that 
 $$  
-f(x+h) + f(x)- \Psi_x(h) \geq  - \|h\| \omega(\|h\|)\ \ \text{whenever}\ \ \ x+h \in G.$$
 Adding these two inequalities, we easily obtain that  $\Psi_x= - \Phi_x$ and
$$  
|f(x+h) - f(x)- \Phi_x(h)| \leq   \|h\| \omega(\|h\|)\ \ \text{whenever}\ \ \ x+h \in G,$$
which implies that $f'(x)= \Phi_x$ and so \eqref{ssapr} holds.
\end{proof}
 Lemma \ref{apr} and Remark \ref{H} immediately imply the following well-known
 result. 
\begin{corollary}\label{aprhil}
Let $X$ be a Hilbert space, $G\subset X$ an open convex set, and let $f$ be a function which is both
 semiconvex and semiconcave with constant $C>0$ on $G$. Then $f$ is Fr\' echet differentiable on $G$ and
\begin{equation}\label{ssaprhil}
|f(y)-f(x)- f'(x)(y-x)| \leq  \frac{C}{2}\|y-x\|^2 ,\ \ \  x, y \in G.
\end{equation}
\end{corollary}
The following result (which clearly implies that the converse of Corollary \ref{aprhil} holds)        is also essentially well-known, but we prefer to present a simple proof.
\begin{lemma}\label{aprc} 
If $X$, $G$, $\omega$ are as in Lemma \ref{apr} and
$f$ satisfies \eqref{ssapr}, then $f$ is both $(
 2\omega)$-semiconvex and $(2\omega)$-semiconcave. Moreover, if $\omega$ is linear,
 then $f$ is both
 $\omega$-semiconvex and $\omega$-semiconcave. 
\end{lemma}
\begin{proof}
For each $x \in G$, $y\in G$, we can apply \eqref{ssapr} for  $x:=y$ and $y:=x$ and obtain
$$|f(x)-f(y)- f'(y)(x-y)| \leq \|x-y\| \, \omega(\|x-y\|).$$
 Adding this inequality with \eqref{ssapr} and using the triangle inequality, we obtain
 $$\vert f'(y)(y-x)-f'(x)(y-x)\vert\leq 2\|y-x\| \omega(\| y-x\|).$$
Since we have proved that \eqref{smder} implies \eqref{hlnap}, we obtain
 that each function $f_{a,v}$ is $C_*^{1,2\omega}$-smooth and consequently
 both $(2\omega)$-semiconvex and $(2\omega)$-semiconcave by \eqref{limpl} (and even
 both $\omega$-semiconvex and $\omega$-semiconcave by Remark \ref{lehimpl} (ii)
 if $\omega$ is linear). So the assertions of the lemma folow by Lemma \ref{pp}.
\end{proof}

We will need also the following  special case of \cite[Lemma 2.7]{DZ2}.
\begin{lemma}\label{skl}
 Let  $\omega \in \M$ and
  $n\geq k\geq2$. Let $\emptyset \neq G \subset \R^n$ be an open convex set, and let
 $L: \R^n \to \R^k$ be a linear surjection. 
 Let $f$ be a function on $L(G)$ which is both $\omega$-semiconvex and $\omega$-semiconcave on $L(G)$.
 Then there exists $C>0$ such that the function  $ f \circ L$ is both $(C\omega)$-semiconvex and $(C\omega)$-semiconcave on 
 $G$.
\end{lemma}
(Note that by well-known facts, $L(G)$ is an open convex set.)

\subsection{Recession cones}
In the proof of our main result, we will use some properties of recession cones (called sometimes also
 asymptotic cones) of unbounded convex sets  $A \subset \R^n$. We refer to the exposition in \cite{Roc},
 however we denote the recession cone of $A$ by $\rec(A)$ (and not $0^+A$ as in \cite{Roc}).
\begin{definition}\label{reccon}
Let $\emptyset \neq A \subset \R^n$ be a convex set. By the recession cone of $A$ we mean
 the set $\rec(A)$ of all $y \in \R^n$ such that $x+ \lambda y \in A$ for each $x \in A$ and $\lambda \geq 0$.
\end{definition}
We will need the following two lemmas on open convex sets  which easily follow from results proved
 in \cite{Roc} for closed convex sets.

\begin{lemma}\label{zvrc}
Let $\emptyset \neq G \subset \R^n$ be an open convex set. Then the following assertions hold.
\begin{enumerate}
\item[(i)]  $\rec(G) = \rec(\overline G)$.
\item[(ii)] $\rec(G)$ is a closed convex cone in $\R^n$ and $0\in \rec(G)$.
\item[(iii)] $\rec(G) = \{ y \in \R^n:\ \exists x \in G\  \forall \lambda \geq 0: 
\ x + \lambda y \in G\}$.
\item[(iv)] If $G$ is unbounded then $\rec(G) \setminus \{0\} \neq \emptyset$.
\item[(v)] $\dim (\rec(G)) = \dim (\spa(\rec(G)))$.  
\item[(vi)] $\dim (\rec(G)) = n$ if and only if $G$ contains a translation of a cone
 with nonempty interior.
\end{enumerate}
\end{lemma}

\begin{proof}
Property (i) follows from \cite[Corollary 8.3.1]{Roc} and (ii) from \cite[Theorem 8.2]{Roc}
 and (i). Theorem \cite[Theorem 8.3]{Roc} and (i) imply (iii). Property (iv) follows
 from \cite[Theorem 8.4]{Roc} and (i). Property (v) is obvious since  $0\in \rec(G)$.
 Using (iii) and (v), it is easy to prove (vi).
\end{proof}
The following lemma is an easy consequence of important \cite[Theorem 9.1]{Roc}.
\begin{lemma}\label{rchl}
Let $ G \subset \R^n$ be an unbounded open convex set and let $L: \R^n \to \R^m$
 be a linear surjection such that $L^{-1} (\{0\}) \cap \rec(G) = \{0\}$.
 Then   $\rec(L(G)) = L(\rec(G))$.
\end{lemma}
\begin{proof}
By Lemma \ref{zvrc} (i) we have $\rec(G) = \rec(\overline G)$. So $L^{-1} (\{0\}) \cap \rec(\overline G) = \{0\}$
 and thus \cite[Theorem 9.1]{Roc} (applied to $C:= \overline G$) implies that $L(\overline G)$
 is closed and $\rec(L(\overline G)) = L(\rec(\overline G))$. It is an easy well-known fact
 that $L(G)$ is an open convex set. Continuity of $L$ implies $L(\overline G) \subset \overline{L(G)}$ and
 so $L(\overline G) = \overline{L(G)}$ since $L(\overline G)$ is closed. Therefore
 $\rec(L(\overline G))= \rec(L(G))$ by Lemma \ref{zvrc} (i) and the equality  $\rec(L(G)) = L(\rec(G))$ follows.
\end{proof}

\section{Linear modulus}\label{lin}

The present short section is devoted to result \eqref{kl1} which is rewritten below as Theorem \ref{lint}. We do not know
 the origin of this interesting useful result which is well-known for a long time.

It is stated (for $G=X$) in \cite[p. 265]{LL}(1986), where it is written (without a reference) that the
 case $G=X=\R^n$ is well-known and  it is shown how the general case  easily  follows from the finite-dimensional case.
 The case $G=X$ is completely proved in \cite{HP}(1989).
Let us note that a weaker version of Theorem \ref{lint} which asserts that $f \in C^{1,1}(G)$
 only, is factually proved in \cite{Rol1}(1979) (for $G=X$ and Lipschitz $f$) and {\it very easily} follows
 (for $G=X$) from \cite[Proposition 10.b]{Mo}(1965) which asserts that if $g$ and $h$ are continuous convex functions such that $g(x)+h(x)= (1/2)\|x\|^2,\ x \in X$, then $g'$ is Lipschitz with constant $1$.  The equality $G=X$ is vital in \cite{Mo} and \cite{HP}.

Several proofs of Theorem \ref{lint} work with second derivatives of $f$. For example,  \cite[Corollary 3.8]{CaSi}
 and  \cite[Exercise 3.8, p. 183]{BC} use distributional derivatives, and \cite[Lemma 2.7]{BPP} and  \cite[Theorem 2.3]{BBEM}
 use Aleksandrov's theorem on a.e. second differentiability of convex functions.

We present below a short elementary proof using the procedure (based on \eqref{buno})
 which will be used also in the proof of Lemma \ref{le_l}. Our proof is similar
 to elementary proofs given in \cite[Theorem 1.1]{Iv} and \cite[Corollary 5.1]{BAJN} and we guess
 that  a similar proof is (essentially) contained also in a much older article.

\begin{theorem}[Well-known theorem]\label{lint}
Let $ X $ be a Hilbert space, $ G\subset X $ an open convex set
 and let a function $f$ on $G$ be both semiconvex and semiconcave with constant $C$.
Then $f$ is differentiable on $G$ and $ f' $ is Lipschitz with constant $C$.
\end{theorem}

\begin{proof}
By Corollary \ref{aprhil} we know that $f'$ exists on $G$ and \eqref{ssaprhil} holds.
First we will show that if
 $ u,v\in G $, $ u\neq v $, then
\begin{equation}\label{koule}
\overline B \bigg(\frac{u+v}{2}, \frac{\|v-u\|}{2}\bigg) \subset G  \ \Longrightarrow\ \|f'(v)-f'(u)\|
 \leq C \|v-u\|.
\end{equation}
To this end, suppose that the assumption of \eqref{koule} holds. Without any loss of generality
 we can suppose that
\begin{equation}\label{buno}
f(u)-f(v)+\frac{1}{2}f'(u)(v-u)-\frac{1}{2}f'(v)(u-v)\leq 0,
\end{equation}
 since otherwise we could work with $\tilde f:= -f$ instead of $f$.
Now consider an arbitrary $ w\in X $, $ \Vert w\Vert=\Vert v-u\Vert $. Then we have
 $\frac{1}{2}(u+v+w) \in G$ and therefore, using  \eqref{ssaprhil} and \eqref{buno}, we obtain
\begin{align*}
f'(v)&\bigg(\frac{1}{2}w\bigg)-f'(u)\bigg(\frac{1}{2}w\bigg)\\
&=f\bigg(\frac{1}{2}(u+v+w)\bigg)-f(u)-f'(u)\bigg(\frac{1}{2}(w+v-u)\bigg)\\
&+f(u)-f(v)+\frac{1}{2}f'(u)(v-u)-\frac{1}{2}f'(v)(u-v)\\
&-\bigg(f\bigg(\frac{1}{2}(u+v+w)\bigg)-f(v)-f'(v)\bigg(\frac{1}{2}(w-v+u)\bigg)\bigg)\\
&\leq\frac{C}{8}\Vert w+(v-u)\Vert^2+\frac{C}{8}\Vert w-(v-u)\Vert^2\\
&=\frac{C}{4}\Vert w\Vert^2+\frac{C}{4}\Vert v-u\Vert^2=\frac{C}{2}\Vert v-u\Vert^2
\end{align*}
and hence
\begin{align*}
\Vert f'(v)-f'(u)\Vert=\sup_{w\in X,\Vert w\Vert=\Vert u-v\Vert}\frac{f'(v)(w)-f'(u)(w)}{\Vert w\Vert}\leq C\Vert v-u\Vert.
\end{align*}
Now consider arbitrary $ x,y\in G $, $ x\neq y $. We can clearly choose $n\in \N$ so large
 that, for each $1\leq i \leq n$,
$$  \overline B \bigg(\frac{z_{i-1}+z_i}{2}, \frac{\|z_i-z_{i-1}\|}{2}\bigg)\subset G,\ \text{where}
 \ z_i:=x+\frac{i}{n}(y-x), \quad i\in\lbrace 0,\dots,n\rbrace.$$
Then \eqref{koule} implies that $\|f'(z_i) - f'(z_{i-1})\| \leq C \|z_i - z_{i-1}\|$
 for each $1\leq i \leq n$, and consequently
\begin{align*}
\Vert f'(y)-f'(x)\Vert&\leq\sum_{i=1}^{n}\Vert f'(z_i)-f'(z_{i-1})\Vert\leq 
  \sum_{i=1}^{n} C \|z_i - z_{i-1}\| =
C\Vert y-x\Vert.
\end{align*}
\end{proof}

\section{Results for a general modulus}\label{gen}
In this section, we prove Theorem \ref{Q} and interesting Corollary \ref{cepr}. Other
 more special versions of implication \eqref{impl} (for $X=\R$ or for special $\omega$) can be found in Remark \ref{specimpl}
 and Remark \ref{oliminf}. 

\begin{remark}\label{jenapr}
Note that in the proof of Lemma \ref{le_l} (and so also in the proofs of 
Corollary \ref{cepr} and Theorem \ref{Q})
 the assumption 
 that $f$ is both $\omega$-semiconvex and $\omega$-semiconcave is applied 
 via property \eqref{ssapr} only.
\end{remark}

Proceeding similarly as in the proof of Theorem \ref{lint}, we obtain the following estimate.
\begin{lemma}\label{le_l}
Let $ X $ be a normed linear space, $ G\subset X $ an open convex set,  $ \omega\in\M $, and let
$f$ be a function on $G$ which  is both $ \omega $-semiconvex and $ \omega $-semiconcave.
If $ y,z\in G $, $ r>0 $ and $ \overline{B}(y,r)\subset G $, then 
\begin{equation}\label{zodh}
\Vert f'(z)-f'(y)\Vert
\leq 2\bigg(1+\frac{\Vert z-y\Vert}{r}\bigg)\omega\bigg(\frac{1}{2}r+\frac{1}{2}\Vert z-y\Vert\bigg).
\end{equation}
\end{lemma}

\begin{proof}
As in the proof of Theorem \ref{lint}, we may suppose that
\begin{align*}
f(y)-f(z)+\frac{1}{2}f'(y)(z-y)-\frac{1}{2}f'(z)(y-z)\leq 0.
\end{align*}
Then for every $ w\in X $, $ \Vert w\Vert=r $, we obtain, using $ \frac{1}{2}(z+y+w)\in G $ and \eqref{ssapr},
\begin{align*}
f'&(z)\bigg(\frac{1}{2}w\bigg)-f'(y)\bigg(\frac{1}{2}w\bigg)\\
&=f\bigg(\frac{1}{2}(z+y+w)\bigg)-f(y)-f'(y)\bigg(\frac{1}{2}(w+z-y)\bigg)\\
&+f(y)-f(z)+\frac{1}{2}f'(y)(z-y)-\frac{1}{2}f'(z)(y-z)\\
&-\bigg(f\bigg(\frac{1}{2}(z+y+w)\bigg)-f(z)-f'(z)\bigg(\frac{1}{2}(w-z+y)\bigg)\bigg)\\
&\leq\frac{1}{2}\Vert w+(z-y)\Vert\omega\bigg(\frac{1}{2}\Vert w+(z-y)\Vert\bigg)+\frac{1}{2}\Vert w-(z-y)\Vert\omega\bigg(\frac{1}{2}\Vert w-(z-y)\Vert\bigg)\\ 
&
\leq (r+\|z-y\|)\Vert\omega\bigg(\frac{1}{2}(r+\|z-y\|)\bigg)
\end{align*}
and hence
\begin{align*}
\Vert f'(z)-f'(y)\Vert&=\sup_{w\in X,\Vert w\Vert=r}\frac{f'(z)(w)-f'(y)(w)}{\Vert w\Vert}\\
&\leq 2\bigg(1+\frac{\Vert z-y\Vert}{r}\bigg)\omega\bigg(\frac{1}{2}r+\frac{1}{2}\Vert z-y\Vert\bigg).
\end{align*}
\end{proof}
Using  \eqref{zodh} for  $z\neq y$ and $r:= \|z-y\|$, we immediately obtain the following result.
\begin{corollary}\label{cepr} 
Let $ X $ be a normed linear space and $ \omega\in\M $.
 Suppose that  a function $ f $ on $X$ is both $ \omega $-semiconvex and $ \omega $-semiconcave. 
Then $ f' $ is uniformly continuous with modulus $ 4\omega $.
\end{corollary}
 \begin{remark}\label{opt4}
Corollary \ref{cepr} is optimal in the sense that the constant $4$ cannot be diminished,
 even for linear modulus. It is factually shown in \cite[Example 5]{BAJN}, where the authors
 consider the example of the function $f(x,y)= x^2-y^2$ on the space $X= (\R^2,\|\cdot\|_{\infty})$.
  They (factually) show that then \eqref{ssapr} holds with $G=X$ and $\omega(t)=t$, which implies
	by  Lemma \ref{aprc} that $f$ is both $\omega$-semiconvex and $\omega$-semiconcave on $X$. Observing
	 that $\|f'((1,1))- f'((0,0))\| \geq 4 \|(1,1)\|_{\infty} = 4$ (see \cite[Example 5]{BAJN}), our assertion follows.
	\end{remark}

\medskip

Using Lemma \ref{le_l}, we give also a relatively short proof
 of Theorem \ref{Q} formulated in  the  introduction.

\begin{proofof}

(i) Suppose that $G$ is bounded.

Let $ x,x+h\in G $ and $ \varepsilon>0 $. We find $ y_0\in G $ and $ r_0>0 $ such that $ \overline{B}(y_0,r_0)\subset G $ and $ \diam G/r_0\leq e_G+\varepsilon $.
Set $ \lambda:=\Vert h \Vert/\diam G $, $ s:=x+\frac{1}{2}h $, $ y:=\lambda y_0+(1-\lambda)s $ and $ r:=\lambda r_0 $. Then 
$$ \overline{B}(y,r)=\lambda\overline{B}(y_0,r_0)+(1-\lambda)\lbrace s\rbrace\subset G\ \  \ \text{and}\ \ \   \frac{\|h\|}{e_G + \ep} \leq r \leq  \frac{\|h\|}{e_G}.  $$
Clearly $\|s-y_0\| \leq \diam G - r_0$ and so 
 $$\|s-y\| = \|\lambda(s-y_0)\| = \frac{\|h\|}{\diam G} \|s-y_0\|\leq 
\|h\| \left( 1- \frac{r_0}{\diam G}\right)         \leq \|h\| - \frac{1}{e_G + \ep} \|h\|.$$
 Consequently, if
 $z=x$ or $z= x+h$, we have  $\|z-y\| \leq \frac{3}{2}\|h\| - \frac{1}{e_G + \ep} \|h\|$.
  Using this inequality, estimates of $r$ and the obvious inequality $e_G\geq 2$, we obtain
	$$  2\left(1+ \frac{\Vert z-y\Vert}{r}\right) \leq 2\left(1 + \frac{3}{2}(e_G+\ep) -1 \right) = 3 (e_G+\ep),\ \text{and}$$
$$ \frac{1}{2}r+\frac{1}{2}\Vert z-y\Vert  \leq
\frac{1}{2}\left(\frac{\|h\|}{e_G} + \frac{3}{2}\, \|h\| - \frac{1}{e_G + \ep}\,  \|h\|\right)
 < \|h\|.$$
 Therefore \eqref{zodh} gives  
$\Vert f'(z)-f'(y)\Vert \leq 3 (e_G+\ep) \omega(\|h\|)$ 
 and  
$$\Vert f'(x+h)-f'(x)\Vert \leq \Vert f'(x+h)-f'(y)\Vert +
\Vert f'(x)-f'(y)\Vert \leq 6 (e_G+\ep) \omega(\|h\|)$$
 which implies the assertion of (i).
\medskip

(ii) Suppose that
 $ a\in X $, $ r>0 $ and $ \lbrace u_n\rbrace_{n\in\N} $ are such that $ \Vert u_n\Vert=n $ and $ B(a+u_n,nr)\subset G $ for each $ n\in\N $. 

Let $ x,x+h\in G $. Then there exists $ n\in\N $ such that $ x,x+h\in\overline{B}(a,n) $. Set $ B:=G\cap B(a,(1+r)n) $. Then $ x,x+h,a+u_n\in B $. Since $ B(a+u_n,rn)\subset B $, we have
\begin{align*}
e_B\leq\frac{\diam B}{rn}\leq\frac{2(1+r)n}{rn}=2\left(1+\frac{1}{r}\right).
\end{align*}
Thus $ \Vert f'(x+h)-f'(x)\Vert\leq 12(1+1/r)\omega(\Vert h\Vert) $ by the case (i).\hfill $ \square $
\end{proofof}

%\medskip
\begin{remark}\label{specimpl}
Let $X$ be a normed linear space, $G \subset X$ an open convex set, $\omega \in \M$, and 
 let $f$ be both $\omega$-semiconvex and $\omega$-semiconcave on $G$.
\begin{enumerate}
\item[(i)]
If $X$ is a Hilbert space and $\omega$ is linear, then 
$f \in C_*^{1,2\omega}(G)$.
\item[(ii)]
If $X=\R$, then $f \in C_*^{1,2\omega}(G)$ too.
\item[(iii)] If $X$ is general and $\omega$ is linear, then $f \in C_*^{1,4\omega}(G)$. 
\end{enumerate}
 Indeed, (i) is a reformulation of \eqref{kl1} by Remark \ref{H}. Easy assertion (ii)
 follows e.g. from  \cite[Proposition 2.8 (i)]{DZ2}.
 To prove (iii), we can proceed   similarly  as in the second part of the proof of Theorem \ref{lint},
 choosing for given $x$, $y$ the number $n$ so large that the points $z_i$
 have the property that $\overline B(z_i, \|z_i-z_{i-1}\|) \subset G$, $i=1,\dots,n$. Then, applying
 \eqref{zodh} to $y:= z_i$, $z:= z_{i-1}$ and $r:= \|z_i-z_{i-1}\|$, the inequality
 $\|f'(y)- f'(x)\|\leq 4 \omega(\|y-x\|)$ easily follows.
\end{remark}

\section{Optimality in finite-dimansional spaces}\label{opt}

By  Remark \ref{specimpl} (iii), we know that if $\omega$ is linear, then the implication
 \eqref{impl} holds for each open convex subset $G$ of an arbitrary normed linear space. This assertion holds also for some nonlinear
 $\omega \in \M$ (cf. Remark \ref{oliminf} below), but the following example shows that for the
 most interesting nonlinear moduli this assertion does not hold.

\begin{example}\label{expas}
Let  $X=\R^2 $ and $ G:=\R\times (0,1)$. Let $\omega \in \M$ be such
 that
\begin{equation}\label{liminf}
 \liminf_{t\to 0+} \frac{\omega(t)}{t} > 0 \ \ \text{and}\ \ 
 \liminf_{t\to \infty} \frac{\omega(t)}{t} = 0.
\end{equation}
 Then the implication \eqref{impl} does not hold.
 \end{example}
 
\begin{proof}
 Set 
$$g(x):= x_1 x_2,\ \ x=(x_1,x_2) \in G.$$
By  \eqref{liminf} we can choose $1>\delta>0$ and $c>0$ such that
 $ ct \leq \omega(t)$  for each $t \in (0, \delta]$. Then the monotonicity of $\omega$ implies that $c \delta \leq  \omega(t)$
 for each $t \in (\delta, \infty)$.
We will show that
\begin{equation}\label{gss}
\text{
 $g$ is both $\left(\frac{2}{c \delta} \omega\right)$-semiconvex
 and $\left(\frac{2}{c \delta} \omega\right)$-semiconcave on $G$.}
\end{equation}
 To this
  end consider $x= (x_1,x_2) \in G$ and $h = (h_1,h_2)\in \R^2$  with
	 $x+h \in G$. Note that  $|h_2|<1$. 
	We have
	$$ |g(x+h) - g(x) - g'(x) (h)|= |(x_1+h_1)(x_2+h_2) - x_1 x_2 
	 - (x_2 h_1+x_1h_2) | = |h_1 h_2|.$$
Now we distinguish two cases. If $\|h\| \leq \delta$, then
$$|g(x+h) - g(x) - g'(x) (h)|  = |h_1 h_2|\leq \|h\|\cdot \|h\|
\leq  \|h\|\, \frac{1}{c} \omega(\|h\|).$$
If  $\|h\| > \delta$, we have
$$|g(x+h) - g(x) - g'(x) (h)|  = |h_1 h_2|\leq |h_1| \leq \|h\|\, 
 \frac{1}{c\delta} \omega(\|h\|).$$
 So Lemma \ref{aprc} implies \eqref{gss}, and consequently the function
 $f:= \frac{c\delta}{2} g$
is both $ \omega$-semiconvex
 and $\omega$-semiconcave on $G$.  

However, we will show that $f \notin C^{1,\omega}(G)$. Suppose to the contrary that $f'$ is uniformly continuous with modulus $D \omega$ on $G$ for some $D>0$.
 Then also  $h(x_1,x_2):= \frac{\partial g}{\partial x_2}(x_1,x_2) = x_1$ is uniformly continuous with modulus $E\omega$ on $G$ for some $E>0$. By \eqref{liminf} there exists $t_0>0$ such that
$ E \omega(t_0)  < t_0$ and therefore
$$ t_0 = h\bigg(t_0,\frac{1}{2}\bigg) = h\bigg(t_0,\frac{1}{2}\bigg) -
 h\bigg(0,\frac{1}{2}\bigg)\leq  E \omega(t_0) < t_0,$$
 which is a contradiction. 
\end{proof}

\begin{remark}\label{onray}
We have proved that if $\omega \in \M$ has properties \eqref{liminf}, then there exists
 a $C^1$ function $f$ on the strip $\R \times (0,1)$ which is both $\omega$-semiconvex and 
 $\omega$-semiconcave, and for which $f'$ is  uniformly  continuous on the ray                
 $[0,\infty) \times \{\frac{1}{2}\}$ with modulus $D \omega$ for no $D>0$.
\end{remark}

\begin{remark}\label{oliminf}
 \begin{enumerate}
\item
If $\liminf_{t\to 0+} \frac{\omega(t)}{t} = 0$, then
 (for arbitrary $X$ and $G$) $f$ is $\omega$-semiconvex (resp.
 $\omega$-semiconcave) on $G$ if and only if $f$ is continuous
 and convex (resp. concave) on $G$ (see e.g.  \cite[p. 3, (2)]{Kr}). Consequently
 \eqref{impl}  clearly holds for arbitrary $X$ and $G$.
\item
 If $\liminf_{t\to \infty} \frac{\omega(t)}{t} > 0$, it is not difficult to 
  prove that \eqref{impl}  holds  for $X=\R^2$ and an arbitrary $G \subset X$.
	\end{enumerate}
It follows that condition \eqref{liminf} in Example \ref{expas} cannot be relaxed.
\end{remark}

The proof of Theorem \ref{hlav}  for $n=2$ is based on the constructions of the following two lemmas. 
 (The case of general $n$ needs a nontrivial inductive argument.)

\begin{lemma}\label{E}
Let $\eta: [0,\infty) \to [0,\infty)$ be a  continuous nondecreasing nonconstant concave function with $\eta(0)=0$ and $\lim_{t\to \infty} \eta(t)/t = 0$.
Then there exists a concave modulus $\omega$ such that
\begin{enumerate}
\item[(i)] $\omega(t)\geq \min(1,t),\ \ t \in [0,\infty),$
\item[(ii)] the function 
 $g(t):= \frac{\eta(t)\cdot \log t}{t\cdot  \omega(t)},\ t \in [1,\infty),$ is bounded, 
\item[(iii)] $\lim_{t\to \infty} \omega(t)/\log t = 0$, 
\item[(iv)] $\lim_{t\to \infty} \omega(t) = \infty$.
\end{enumerate}
\end{lemma}

\begin{proof}
 First we will construct a concave modulus $\tilde \omega \in \M$, for which
 conditions (i)-(iii) hold.
Set  $\tilde\omega(t):=t,\ \ t \in [0,1]$,\  and 
$$ \tilde\omega(t):= 1 + \frac{1}{\eta(1)} \cdot \int_1^t  \frac{\eta(u)}{u^2} \ du,\ \ t\geq 1.$$
Then clearly $\tilde\omega \in \M$ and (i) holds for $\tilde\omega$.

Note that the function  $u \mapsto \frac{\eta(u)}{u}$ is nonincreasing on $[1,\infty)$ by concavity of $\eta$.
So the function $u \mapsto \frac{\eta(u)}{u^2}$ is nonincreasing   on $[1,\infty)$ too. Consequently
 we obtain that $\tilde\omega$ is concave on $[1,\infty)$ and since
 $\tilde\omega_+'(1)=1$, we obtain  that $\tilde\omega$ is concave on $[0,\infty)$.

Since $u \mapsto \frac{\eta(u)}{u}$ is nonincreasing on $[1,\infty)$, we obtain
 that, for  $1<u\leq t$, we have
$$\frac{\eta(u)}{u^2} \geq \frac{\eta(t)}{t} \cdot  \frac{1}{u},$$
consequently
$$ \tilde\omega(t) \geq \frac{1}{\eta(1)} \cdot \frac{\eta(t)}{t} \cdot \int_1^t \frac{1}{u}\ du = \frac{1}{\eta(1)} \cdot \frac{\eta(t)}{t} \cdot \log t,\ \ t\geq 1,$$
and so (ii) for $\tilde\omega$ follows. 

Finally, (iii) for $\tilde\omega$ follows by  L'Hospital's rule, since
$$\lim_{t\to \infty} \frac{\int_1^t  \frac{\eta(u)}{u^2} \ du}{\log t} =  \lim_{t\to \infty}  \frac{\frac{\eta(t)}{t^2}}{\frac{1}{t}}=    0.$$
Now it is easy to see that the modulus
$$\omega(t):= \tilde\omega(t) + \sqrt{\log (1+ t)},\ \ t \in [0,\infty],$$
 has all demanded properties.
 \end{proof}

\begin{lemma}\label{zaklmn}
Let $ \eta:[0,\infty)\to [0,\infty) $ be a nondecreasing nonconstant continuous concave function such that $\eta(0) =0$ and $\lim_{t \to \infty} \eta(t)/t = 0$. 
Set $$ H:=\lbrace(x_1,x_2)\in\R^2:x_1>1,\;\vert x_2\vert<\eta(x_1)\rbrace $$ and
\begin{align*}
f(x_1,x_2):=\log(x_1)x_2,\quad (x_1,x_2)\in H.
\end{align*}
Then there exists a concave $ \omega\in\M $ such that
\begin{equation}\label{vlom}
 \lim_{t \to \infty} \omega(t)= \infty,\ \ \ \ \lim_{t \to \infty} \omega(t)/\log t =0,
\end{equation}
 and 
\begin{equation}\label{amb}
\text{$f$ is both $(C\omega)$-semiconvex and $(C\omega)$-semiconcave 
 for some $C>0$.}
\end{equation}
\end{lemma}

\begin{proof}
By Lemma \ref{E} there exists a concave $ \omega\in\M $ such that
\eqref{vlom}  and conditions  (i), (ii) of Lemma \ref{E} hold. So we
 can find $D>0$ such that
\begin{equation}\label{od}
\eta(t)\log(t)\leq Dt\omega(t)\ \ \ \text{and}\ \ \ \eta(t)\leq Dt,\quad t \geq 1.
\end{equation}
To prove \eqref{amb}, we will use
 Lemma \ref{post}.

So we consider arbitrary 
$x=(x_1,x_2) \in H$,  $y=(y_1,y_2) \in H$ and set $h= (h_1,h_2):= y-x$. By the symmetry,
 we can suppose that $h_1 \geq 0$. 
Now we will estimate from above 
$$ V:= \vert f'(y)(y-x)-f'(x)(y-x)\vert = \vert f'(x+h)(h)-f'(x)(h)\vert.$$
An elementary computation gives
\begin{align*}
V&= \left| \frac{x_2+h_2}{x_1+  h_1}\cdot h_1 + \log(x_1+h_1) h_2 - \frac{x_2}{x_1}\, h_1 
 - \log(x_1) h_2\right|\\
&\leq h_1 \, \left| \frac{x_2+h_2}{x_1+  h_1} - \frac{x_2}{x_1}\right| + |h_2|\,
( \log(x_1+h_1)- \log(x_1)) =: A + B.
\end{align*}
We have
$$ A=  h_1 \, \left|\frac{h_2x_1-h_1x_2}{(x_1+h_1) x_1}\right| 
 \leq \frac{h_1 |h_2|}{x_1+h_1} + \frac{(h_1)^2|x_2|}{(x_1+h_1) x_1}.$$

Using Lemma \ref{E} (i) and $x_1 >1$, we obtain

$$ \frac{h_1 |h_2|}{x_1+h_1} \leq \min(\|h\|,\|h\|^2) = \|h\| \min(1,\|h\|)
 \leq \|h\| \omega(\|h\|)$$
 and, using also $|x_2| \leq \eta(x_1) \leq D x_1$,
$$ \frac{(h_1)^2|x_2|}{(x_1+h_1) x_1}\leq D \,
\frac{(h_1)^2}{x_1+h_1} \leq D \, \min(\|h\|,\|h\|^2) \leq 
 D\, \|h\| \omega(\|h\|),$$
 and consequently  $A \leq (D+1)\, \|h\| \omega(\|h\|)$.
\medskip

If $h_1\leq 1$, then $h_1 \leq \omega(h_1)$ by Lemma \ref{E} (i), and therefore
\begin{align*}
B= \vert h_2\vert \log\bigg(1 + \frac{h_1}{x_1}\bigg) 
\leq\vert h_2\vert\frac{h_1}{x_1}\leq\vert h_2\vert h_1\leq\vert h_2\vert\omega(h_1)\leq \|h\|
\omega(\| h\|).
\end{align*}

If  $1\leq h_1\leq x_1$,  note that Lemma \ref{E} (i) implies $1 \leq \omega(h_1)$
 and obviously $|h_2| \leq 2\eta(x_1+h_1)$. Therefore we obtain, using also \eqref{od},
\begin{align*}
B&= \vert h_2\vert \log\bigg(1 + \frac{h_1}{x_1}\bigg) 
\leq 2\eta(x_1+h_1)\frac{h_1}{x_1}\leq 2\frac{\eta(2x_1)}{x_1}h_1\\
&\leq 4Dh_1\leq 4Dh_1\omega(h_1)\leq 4D\| h\|\omega(\|h\|).
\end{align*}
If  $x_1\leq h_1$, then we obtain, using $|h_2| \leq 2\eta(x_1+h_1)$, \eqref{od} and concavity of
 $\omega$,
\begin{align*}
B&=\vert h_2\vert(\log(x_1+h_1)-\log(x_1))\\
&\leq 2\eta(x_1+h_1)\log(x_1+h_1)\leq 2D(x_1+h_1)\omega(x_1+h_1)\\
&\leq 4Dh_1\omega(2h_1) \leq 8Dh_1\omega(h_1)\leq 8D\|h\|\omega(\|h\|).
\end{align*}
Using the above estimates, we obtain  $V= A+B \leq (9D+1)\, \|h\| \omega(\|h\|) $.
 So Lemma \ref{post} implies that \eqref{amb} holds for $C:= 9D+1$.
\end{proof}

The proof of our main result in general $\R^n$ use an inductive
 argument which needs the following notion.

\begin{definition}\label{gno}
Let $n \in \N$ and $\omega \in \cal M$.
 Then we denote by $\cal G_n(\omega)$ the set of all open convex sets
 $G\subset \R^n$ for which there exist a ray  $\rho \subset G$, $C>0$ and a function $f\in C^1(G)$
which is both $(C\omega)$-semiconvex and  $(C\omega)$-semiconcave and for which 
 $f'$ is  uniformly continuous on $\rho$ with modulus $D\omega$
 for no $D>0$.
\end{definition}

\begin{remark}\label{gimpl}
If $G \in \cal G_n(\omega)$, then the implication \eqref{impl} does not hold. Indeed,
 then the function $f^*:= (1/C) f$ (where $C$ and $f$ are as in Definition \ref{gno}) is both $\omega$-semiconvex and  $\omega$-semiconcave 
 but clearly $f^* \notin C^{1, \omega}(G)$.
 \end{remark}

\begin{remark}\label{pasg}
Note that Remark \ref{onray} immediately implies that, if $\omega\in \M$ has properties
 \eqref{liminf}, then
$\R \times (0,1) \in \cal G_2(\omega)$.
\end{remark}

The following result on  further  unbounded open convex subsets in $\R^2$ easily follows from
 Lemma \ref{zaklmn}.
 
\begin{lemma}\label{ul}
 Let $u$ be 
	 a positive  nondecreasing concave function on $(1,\infty)$ and let $l$ be a negative
	nonincreasing convex function on  $(1,\infty)$
	 such that  $\lim_{x\to \infty} u(x)/x =0$  and  $\lim_{x\to \infty} l(x)/x =0$.
	Then there exists a concave modulus $\omega \in \M$ such that $\lim_{t\to \infty} \omega(t)=
	 \infty$ and
	\begin{equation}\label{ulgo}
	 G:= \{(x,y):\ x>1,\ l(x) < y < u(x)\} \in \cal G_2(\omega).
	\end{equation}
	\end{lemma}
	
\begin{proof}
Set  $h(x):= u(x)- l(x),\ x>1$, $a:= \max(h(2), h'_+(2))$ and
 $b:= 2a- h(2)$. Clearly $a>0$, $b>0$.
Let $\eta(x):= ax$ for  $x \in [0,2]$ and $\eta(x):= h(x)+b$ for
 $x \in (2,\infty)$. Then $\eta$ is clearly a nonnegative  
 nonconstant nondecreasing continuous concave function on $[0, \infty)$ such that 
  $\eta(0)=0$, 
  $\lim_{x\to \infty} \eta(x)/x =0$ and
	$$G \subset H:=     
\lbrace(x_1,x_2)\in\R^2:x_1>1,\;\vert x_2\vert<\eta(x_1)\rbrace.$$
Let $f(x_1,x_2):=\log(x_1)x_2,\ (x_1,x_2)\in H$,  and let
 $ \omega\in\M $  be a modulus which corresponds to $\eta$ and $f$
 by Lemma \ref{zaklmn}. 
Suppose to the contrary $G \notin \cal G_2(\omega)$.
Since $G\subset H$ and \eqref{amb} implies 
 that $f$ is both $(C\omega)$-semiconvex and $(C\omega)$-semiconcave 
 on $G$ for some $C>0$, we obtain that 
 $f'$ is uniformly continuous with modulus $D \omega$ on the ray 
$\rho:= [2,\infty) \times \{0\}\subset G$ for some $D>0$.
 Consequently
$$\bigg|\frac{\partial f}{\partial x_2}(t,0)- \frac{\partial f}{\partial x_2}(2,0)\bigg| = |\log t - \log 2| \leq D \omega(t-2),\ \ t>2,$$
 which clearly contradicts \eqref{vlom}.
\end{proof}

The following lemma is geometrically obvious.
\begin{lemma}\label{K}
Let  $V \subset \R^n$ be a two-dimensional Hilbert space and $C\subset V$ a two-dimensional closed convex cone such that $\{v,-v\} \subset C$ for no $v\neq 0$. Then there exists  $0\neq w \in V$ such that 
$$  \spa(\{w\}) \cap C = \{0\}.$$
\end{lemma}

\begin{lemma}\label{L}
 Let  $\omega \in \M$ be concave.
 Let $n\geq k\geq2$, let $G \subset \R^n$ be an open convex set, and let
 $L: \R^n \to \R^k$ be a linear surjection. Suppose
 that  $H:= L(G) \in \cal G_{k}(\omega)$ and 
$ \rec(H) \subset L(\rec(G)) $.
 Then $G \in \cal G_n(\omega)$. 
\end{lemma}

\begin{proof}
 Since  $H \in \cal G_{k}(\omega)$, there exist a ray  $\rho \subset H$ and a function $f\in C^1(H)$
which is both $\omega$-semiconvex and  $\omega$-semiconcave and for which 
 $f'$ is  uniformly continuous on $\rho$ with modulus $D\omega$
 for no $D>0$.
Set
$$ \tilde f(y):= f(L(y)),\ \ y \in G.$$
Then  $\tilde f \in C^1(G)$ and there exists $C>0$ such that 
 $\tilde f$ is both $(C\omega)$-semiconvex and  $(C\omega)$-semiconcave 
on $G$ by Lemma \ref{skl}.

We can write  $\rho= \{a+ t\, v:\ t \in [0,\infty)\}$,
 where $a\in H$ and $ v \in \R^{k} \setminus \{0\}$. Then 
 $v \in \rec(H)$ by Lemma \ref{zvrc} (iii)  and so, by the assumptions, there exists
 $\tilde v \in \rec(G)$ with $L(\tilde v) =v$. Choose  $\tilde a \in G$                    
 with  $L(\tilde a)=a$. Then  $\tilde \rho:= \{\tilde a + t\, \tilde v:\ t \in [0,\infty)\} \subset G$.  It is
 easy to see that  $L(\tilde \rho) = \rho$ and, for $\alpha:= \|\tilde v\|/\|v\|$,
\begin{equation}\label{roroh}
\alpha \cdot \|L(\tilde x) - L(\tilde y)\| = \|\tilde x-\tilde y\|\ \ \text{for each}\ \ 
 \tilde x, \tilde y \in \tilde \rho.
\end{equation}

Now suppose, to the contrary, that $G \notin \cal G_n(\omega)$. 
Then there exists $\tilde D>0$, such that
\begin{equation}\label{Dst}
\| (\tilde f)'(\tilde x) - (\tilde f)'(\tilde y)\| \leq \tilde D\, \omega(\|\tilde x-\tilde y\|),\ \ \ \tilde x, \tilde y \in \tilde \rho.
\end{equation}
Consider the dual mapping $L^*: (\R^k)^* \to (\R^n)^*$ defined by $L^*(\vf):= \vf \circ L$,
 $\vf \in  (\R^k)^*$. Since $L$ is surjective, $L^*$ is injective and so 
 $L^*: (\R^k)^* \to  L^*((\R^k)^*)$ is a linear isomorphism. So there exists $\beta>0$ 
 such that
\begin{equation}\label{dub}
 \beta \|\vf \circ L\| \geq \|\vf\|\ \ \text{for each}\ \ \vf \in (\R^k)^*.
\end{equation}
Now consider arbitrary  $x,y \in \rho$ and find $\tilde x,\tilde y \in \tilde \rho$ such that
 $L(\tilde x)= x$, $L(\tilde y)=y$.  Using \eqref{dub} for $\vf:= f'(x) - f'(y)$, \eqref{Dst},
\eqref{roroh} and the concavity of $\omega$, we obtain
\begin{align*}
\| f'(x)- f'(y ) \| &\leq \beta \| (f'(x)- f'(y )) \circ L \|=  
\beta \|( f'(L(\tilde x))- f'(L(\tilde y)) \circ L\| \\
&= \beta 
\| (\tilde f)'(\tilde x) - (\tilde f)'(\tilde y)\|\leq \beta  \tilde D \omega(\|\tilde x-\tilde y\|)\\
&= \beta  \tilde D \omega(\alpha \|x-y\|) \leq \beta  \tilde D \max(1,\alpha) \omega(\|x-y\|).
\end{align*}
So $f'$ is uniformly continuous on $\rho$ with modulus
 $D \omega$ for $D:= \beta  \tilde D \max(1,\alpha)  $, which
 is a contradiction.
\end{proof}

As an easy consequence we obtain the following fact.
\begin{lemma}\label{A}
 Let  $\omega \in \M$ be concave. Let $n\geq 2$, $G \in \cal G_n(\omega)$,   and let $A: \R^n \to \R^n$ be an affine bijection. Then  $A(G) \in \cal G_n(\omega)$.
\end{lemma}

\begin{proof}
If $A$ is linear, then the assertion easily follows from Lemma \ref{L}. If $A$ is a translation, then the proof is obvious. So the lemma
 clearly follows.
 \end{proof}

\begin{proposition}\label{P}
Let $n\geq 2$ and let $G \subset \R^n$ be an unbounded 
open convex set which contains no translation of a convex cone with nonempty  interior.
Then there exists a concave modulus $\omega \in \cal M$ with
 $\lim_{t \to \infty} \omega(t) = \infty$ such that
  $G \in \cal G_n(\omega)$.
\end{proposition}

\begin{proof}
 We will proceed by induction on $n$.

a) Let $n=2$. Then our assumptions and Lemma \ref{zvrc} (iv), (vi) imply that $\dim (\rec(G)) = 1$. On account of Lemma \ref{A}   we can suppose $(1,0) \in \rec(G)$.
 Now we will distinguish two cases.

 If also $-(1,0) \in \rec(G)$, then clearly $G = \R \times H$ for an
 open convex set $H\subset \R$. Since $\dim (\rec(G)) = 1$, $H$ 
  is a bounded open interval. By  Lemma \ref{A} we can suppose $H=(0,1)$ and so $G \in \cal G_2(\omega)$  e.g. for $\omega(t) = \sqrt t$ by  Remark \ref{pasg}.
	
	If $-(1,0) \notin \rec(G)$,  then  we can suppose (by Lemma \ref{A}) that 
	$\{x\in \R: (x,0) \in G\} = (1,\infty)$. We can also suppose that
	 a support line to $\overline G$ at $(1,0)$ is parallel to the $y$-axis. For $x>1$, set  $u(x): = \sup\{y: (x,y) \in G\}$   and
	$l(x): = \inf\{y: (x,y) \in G\}$. It is easy to show that $u$ is
	 a positive nondecreasing concave function on $(1,\infty)$, $l$ is a negative nonincreasing
	 convex function on  $(1,\infty)$ and 
	$$ G= \{(x,y):\ x>1,\ l(x) < y < u(x)\}.$$
  Then  $A:=\lim_{x\to \infty} u(x)/x =0$, since otherwise
	 $(1,A/2) \in \rec(G)$ and thus $\dim (\rec(G))=2$. Similarly
	 we obtain $\lim_{x\to \infty} l(x)/x =0$. 
	 So our assertion follows from Lemma \ref{ul}.

	b) Now suppose that $n\geq 3$  is given and that the proposition
	 ``holds for $n-1$''. We will distinguish two cases.
	\smallskip
	
	\ \ b1) There exists $v\neq 0$ such that $\{v,-v\} \subset  \rec (G)$ (i.e., $G$ contains a line). 
	
	By Lemma \ref{A}, we can suppose that $v= (0,\dots,0,1)$ and thus
	 there exists an open convex set $H \subset \R^{n-1}$ such that $G= H \times \R$.
	Then clearly $\dim (\rec(H)) < n-1$ and so the inductive assumption gives that $H \in \cal G_{n-1}(\omega)$ for some concave  $\omega \in \cal M$ with
 $\lim_{t \to \infty} \omega(t) = \infty$.    Let $L(x_1,\dots,x_n):= (x_1,\dots,x_{n-1})$. 
Then clearly $L(G)=H$
	 and   $ \rec(H) \subset L(\rec(G)) $.
 Therefore $G \in \cal G_n(\omega)$ by Lemma \ref{L}. 
	
	\smallskip
	
	\ \ b2) There is no $v\neq 0$ such that $\{v,-v\} \subset  \rec (G)$.
	
	We will show that there exists  a linear surjection  $L: \R^n \to \R^{n-1}$ such that
\begin{equation}\label{ol1}
L^{-1} (\{0\}) \cap \rec(G) = \{0\}\ \  \text{and}
\end{equation}
\begin{equation}\label{ol2}
\ \ \dim (L(\rec(G))) < n-1.
\end{equation}
If $\dim(\rec(G)) < n-1$, it is easy  (using Lemma \ref{zvrc} (v)) to find $L$ satisfying 
 \eqref{ol1} and observe that \eqref{ol2}  holds.

In the opposite case Lemma \ref{zvrc} (v),(vi) imply that $\dim(\rec(G)) = n-1$ and there exist  linearly independent vectors
 $v_1,\dots,v_{n-1} \in \rec(G)$. Now we apply  Lemma \ref{K}
 to the space $V:= \spa \{v_1,v_2\}$ and the cone  $C:=V \cap\,  \rec(G)$ (which is closed by Lemma \ref{zvrc} (ii))
 and obtain a corresponding vector $w$. Then any linear surjection  $L: \R^n \to \R^{n-1}$ 
 with $L(w)=0$  clearly satisfies \eqref{ol1}.  Moreover, $L(w)=0$ implies that
 the vectors $L(v_1)$ and $L(v_2)$ are linearly dependent and 
consequently $\dim ( L(\rec(G))  ) < n-1$.

So there exists $L$ for which \eqref{ol1} and \eqref{ol2} hold.
 Using \eqref{ol1} and Lemma \ref{rchl} we obtain
 $\rec(L(G)) = L(\rec(G))$. So $\dim (\rec(L(G)) ) < n-1$ and consequently 
 $L(G) \in \cal G_{n-1}(\omega)$  for some  concave  $\omega \in \cal M$ with
 $\lim_{t \to \infty} \omega(t) = \infty$   by the induction assumption. 
 Therefore we obtain
 $G \in \cal G_{n}(\omega)$ by   
 Lemma \ref{L}.
\end{proof}

\section{Consequences and open questions}\label{posl}
As already mentioned,  our results are very closely related to  a first-order quantitative
 converse  Taylor theorem and to the problem whether 
$f\in C^{1,\omega}(G)$  whenever $f$ is continuous and smooth in a corresponding sense on all lines.

To describe precisely these relations, first
recall \cite[Definition 121]{HJ} in the special case of real-valued functions and
 $p=1$ (which deal with  approximation of a function by polynomials of  degree
 $p\leq 1$).
\begin{definition}\label{UT}
Let $X$ be a normed linear space and $G\subset X$ an open set. We say that a function
 $f: G \to \R$ is $UT^1$-smooth on $G$ with modulus $\omega \in \M$ (or shortly $UT^1_{\omega}$-smooth) if for each $x \in G$
 there exist $b \in \R$ and $x^* \in X^*$  such that
\begin{equation}\label{bxs}
  |f(x+h) - (b+ x^*(h))| \leq \|h\|\, \omega(\|h\|)\ \ \ \text{whenever}\ \ \ x+h \in G.
	\end{equation}
	\end{definition}
	
	\begin{remark}\label{prf}
Note that if \eqref{bxs} holds then clearly $b=f(x)$ and $x^*= f'(x)$.
So $f$ is $UT^1_{\omega}$-smooth on $G$ if and only if 
 \eqref{ssapr} holds.
\end{remark}
Further, we will use the following terminology.  
\begin{definition}\label{smol}
Let $X$ be a normed linear space, $G\subset X$ an open convex set,
 $\omega \in \M$, and let $f$ be a function on $G$. 
 We will say that $f$ is $C^{1,\omega}_*$-smooth on all lines if,
 for every $a \in G$ and $v\in X$, $\|v\|=1$, the function
 $f_{a,v}: t \mapsto f(a+tv)$ defined on the open interval
 $\{t\in \R:\ a+tv \in G\}$ is $C^{1,\omega}_*$-smooth.
\end{definition}
Using this notation, we can reformulate some already mentioned facts as follows.
\begin{proposition}\label{skoekv}
Let $X$ be a normed linear space, $G\subset X$ an open convex set,
 $\omega \in \M$, and let $f$ be a continuous function on $G$.
Then the following implications hold.
\begin{enumerate}
\item[(i)] If $f$ is both $\omega$-semiconvex and $\omega$-semiconcave, then $f$ is $UT^1_{\omega}$-smooth.
\item[(ii)] If $f$ is $UT^1_{\omega}$-smooth, then 
 $f$ is both $(2\omega)$-semiconvex and $(2\omega)$-semiconcave.
\item[(iii)] If $f$ is $C^{1,\omega}_*$-smooth on all lines,
 then $f$ is both $\omega$-semiconvex and $\omega$-semiconcave.
\item[(iv)] If $f$ is both $\omega$-semiconvex and $\omega$-semiconcave, then $f$ is $C^{1,2\omega}_*$-smooth on all lines.
\end{enumerate}
\end{proposition}

\begin{proof}
Assertions (i) and (ii) are reformulations of Lemma \ref{apr} and Lemma \ref{aprc}, respectively.
Using \eqref{limpl} and Lemma \ref{pp} we obtain (iii). Lemma \ref{pp} together with
 Remark  \ref{specimpl} (ii)  give (iv).
\end{proof}
\begin{remark}\label{obrac}
If $\omega$ in Proposition \ref{skoekv} is linear, then the converse of (i) holds by Lemma
 \ref{aprc} and the converse of (iv) holds by Remark \ref{lehimpl} (ii) and Lemma \ref{pp}.
\end{remark}

Using Proposition \ref{skoekv}, Remark \ref{jenapr} and Remark \ref{prf}, we easily see that
\medskip

{\it Theorem \ref{Q} and Corollary \ref{cepr} remain true if we replace the property
 ``$f$ is both
 $\omega$-semiconvex and $\omega$-semiconcave'' by the 
%(sligtly weaker)
 property that ``$f$ 
is $UT^1$-smooth function with modulus
 $\omega$'' or by the 
%(sligtly stronger)
 property that
``$f$ is continuous and $C^{1,\omega}_*$-smooth on all lines''.} 
 \medskip

So we obtain  in the very special case $p=1$, $Y=\R$ (and finite modulus $\omega$) a more precise
 qualitative versions of  converse Taylor theorem \cite[Corollary 126]{HJ} and also of 
 \cite[Proposition 12]{JZ} and \cite[Theorem 127]{HJ} (which deal with functions which are
 ``$C^{p,\omega}_*$-smooth on all lines'').

In particular, we obtain the following version of Corollary \ref{cepr}.

\begin{corollary}\label{cepr2} 
Let $ X $ be a normed linear space and $ \omega\in\M $.
 Suppose that  a continuous function $ f $ on $X$ is 
 $C_*^{1,\omega}$-smooth on all lines.
Then $f$ is $C_*^{1,4\omega}$-smooth on $X$.
\end{corollary}

 Further, Proposition \ref{skoekv}  shows that our basic result Theorem \ref{hlav} can be reformulated in the following
 form.  

\begin{theorem}\label{hlavp}
Let $G \subset \R^n$ ($n\geq 2$) be an unbounded open convex set which does not contain a translation of a convex cone with nonempty interior. Then there exists a concave modulus $\omega \in \M$
 with $\lim_{t \to \infty} \omega(t) = \infty$ and a continuous function $f$
 such that
\begin{enumerate}
\item[(i)]
$f$ is $C^{1,\omega}_*$-smooth on all lines,
\item[(ii)]
$f$ is $UT^1$-smooth function with modulus
 $\omega$ and
 \item[(iii)]
  $f \notin C^{1,\omega}(G)$.
\end{enumerate} 
\end{theorem}
In particular, the assumptions on $G$ of   \cite[Corollary 126]{HJ} 
 cannot be relaxed in the case $X= \R^n$.
\medskip

Finally observe that, by Proposition \ref{skoekv}, equivalence \eqref{chardve} can be extended
 as follows.
\begin{proposition}\label{charfor}
Let $X$ and $G$    be as in Theorem \ref{HJ}. Then, for each $\omega \in \M$ and for each function $f$ on $G$, the following conditions are equivalent.
\begin{enumerate}
\item
$f \in C^{1,\omega} (G).$
\item
$f\  \text{is both $(K\omega)$-semiconvex and $(K\omega)$-semiconcave}\ \text{for some}\ K>0.$
\item
 $f$ is $UT^1_{K\omega}$-smooth for some $K>0$.
\item
$f$ is continuous and $C_*^{1,K\omega}$-smooth on all lines for some $K>0$. 
 \end{enumerate}
\end{proposition}

The following natural (cf. Remark \ref{oliminf}) question remains open.

\begin{question}\label{jedna}
Let $\omega \in \M$, $n\geq 2$ and let $G \subset \R^n$ be an open convex set
which contains no translation of a convex cone
 with nonempty interior. Does there exist a function $f$ on $G$ which is both $\omega$-semiconvex and $\omega$-semiconcave and does not belong to $C^{1,\omega}(G)$, if
\begin{enumerate}
\item
 $\liminf_{t\to 0+} \frac{\omega(t)}{t} > 0$ \ and \ 
 $\liminf_{t\to \infty} \frac{\omega(t)}{t} = 0$?
\item
$\omega$ is concave,\ $\liminf_{t\to 0+} \frac{\omega(t)}{t} > 0$ \ and \ 
 $\liminf_{t\to \infty} \frac{\omega(t)}{t} = 0$?
\item
 $\omega(t)= t^{\alpha},\ t\geq 0,$ for some  $0< \alpha<1$?
\item
$\omega(t)= \sqrt t,\ t\geq 0$? 
\end{enumerate}
\end{question}


\begin{thebibliography}{WWW}

\bibitem[BAJN]{BAJN}
G.O. Berger, P.A. Absil, R.M. Jungers, Y. Nesterov:
On the quality of first-order approximation of functions with H\" older continuous gradient,
J. Optim. Theory Appl. 185 (2020) 17--33. 

\bibitem[BBEM]{BBEM}
M. Ba\v c\' ak, J.M. Borwein, A. Eberhard, B.S. Mordukhovich:
Infimal convolutions and Lipschitzian properties of subdifferentials for prox-regular functions in Hilbert spaces, J. Convex Anal. 
17 (2010) 737--763. 

\bibitem[BC]{BC}
M. Bardi, I. Capuzzo-Dolcetta:
Optimal control and viscosity solutions of Hamilton-Jacobi-Bellman equations,
Birkhäuser Boston, Boston, MA (1997).

\bibitem[BPP]{BPP}
M. Bougeard, J.-P. Penot, A. Pommellet:
Towards minimal assumptions for the infimal convolution regularization,
J. Approx. Theory 64 (1991) 245--270. 

\bibitem[CaSi]{CaSi}
P. Cannarsa, C. Sinestrari:
Semiconcave functions, Hamilton-Jacobi equations, and optimal control, 
Progress in Nonlinear Differential Equations and their Applications, 58. Birkh\"auser Boston, Inc., Boston, MA (2004).

\bibitem[DZ1]{DZ1} 
J.~Duda, L.~Zaj\'\i\v{c}ek:
Semiconvex functions: representations as suprema of smooth functions and extensions,
J. Convex Analysis 16 (2009) 239--260.  


\bibitem[DZ2]{DZ2}
J.~Duda, L.~Zaj\'\i\v{c}ek:
Smallness of singular sets of semiconvex functions in separable Banach spaces,
J. Convex Anal. 20 (2013) 573--598.



\bibitem[HJ]{HJ}
P. H\'ajek, M. Johanis: Smooth Analysis in Banach Spaces,
Walter de Gruyter, Berlin (2014).

\bibitem[HP]{HP}J.-B. Hiriart-Urruty, P. Plazanet: Moreau's decomposition theorem revisited,    
 Ann. Inst. H. Poincar\' e Anal. Non Lin\' eaire 6 (1989), suppl. 325--338. 


\bibitem[Iv]{Iv}
G.E. Ivanov:
Continuity of optimal controls in differential games, and some properties of weakly and strongly convex functions (Russian),
Mat. Zametki 66 (1999) 816--839; translation in
Math. Notes 66 (1999) 675--693. 


%MR1019120 (91b:58017) Reviewed
%Hiriart-Urruty, J.-B. (F-TOUL3); Plazanet, Ph.
%Moreau's decomposition theorem revisited.
%Analyse non linéaire (Perpignan, 1987).
%Ann. Inst. H. Poincaré Anal. Non Linéaire 6 (1989), suppl., 325–338. 

\bibitem[Jo]{Jo}
M. Johanis:
A quantitative version of the Converse Taylor theorem: $C^{k,\omega}$-smoothness,
Colloq. Math. 136 (2014) 57--64.

\bibitem[JTZ]{JTZ}
A. Jourani, L. Thibault, D. Zagrodny:
$C^{1,\omega(\cdot)}$-regularity and Lipschitz-like properties of subdifferential,
Proc. Lond. Math. Soc. (3) 105 (2012) 189--223.

\bibitem[JZ]{JZ}
M. Johanis, L.~Zaj\'\i\v{c}ek:
Smoothness via directional smoothness and Marchaud's theorem in Banach spaces,
J. Math. Anal. Appl. 423 (2015)  594--607. 

\bibitem[Kr]{Kr} V. Kry\v stof:
Generalized versions of Ilmanen lemma: Insertion of $ C^{1,\omega} $ or $ C^{1,\omega}_{\loc} $ functions,
Comment. Math. Univ. Carolin. 59 (2018) 223--231.

\bibitem[Kr2]{Kr2} V. Kry\v stof:
Semiconvex functions and their differences, Bachelor thesis,
Charles University, Prague (2013) (Czech).

\bibitem[LL]{LL}
J.M. Lasry, P.L. Lions:
A remark on regularization in Hilbert spaces,
Israel J. Math. 55 (1986) 257--266. 

\bibitem[Mo]{Mo}
J.-J. Moreau:
Proximit\' e et dualit\' e dans un espace hilbertien,
Bull. Soc. Math. France 93 (1965) 273--299. 

\bibitem[Roc]{Roc}
R.T. Rockafellar: 
Convex analysis,
Princeton Mathematical Series, No. 28, Princeton University Press, Princeton, N.J. (1970).

\bibitem[Rol1]{Rol1}
S. Rolewicz:
On $\gamma$-paraconvex multifunctions,
Math. Japon. 24 (1979/80) 293--300. 

\bibitem[Rol2]{Rol2} 
S. Rolewicz: 
{On the coincidence of some subdifferentials in the class of $\alpha (\cdot)$-paraconvex functions}, 
Optimization {50} (2001) 353--360.
 
\bibitem[Rol3]{Rol3}
S. Rolewicz: 
{An extension of Mazur's theorem on G\^ateaux differentiability to the class of strongly $\alpha(\cdot)$-paraconvex functions}, 
Studia Math. {172} (2006) 243--248. 

\bibitem[Za]{Za}
L.~Zaj\'\i\v{c}ek: 
Differentiability of approximately convex, semiconcave and strongly paraconvex functions,
J. Convex Analysis 15 (2008) 1--15. 

\end{thebibliography}
\end{document}